\documentclass[oneside,english]{amsart}
\usepackage[T1]{fontenc}
\usepackage[latin1]{inputenc}
\usepackage{indentfirst}
\usepackage{amsmath}
\usepackage{yhmath} 
\usepackage{amsfonts}
\usepackage{amssymb}
\usepackage{textcomp} 
\usepackage{epic}
\usepackage{verbatim} 
\usepackage[english]{babel}
\usepackage{hyperref}

\def \Ind #1#2{#1\setbox 0=\hbox{$#1x$}\kern\wd 0 \hbox to 0pt{\hss$#1\mid$\hss}
\lower.9\ht 0\hbox to 0pt{\hss$#1\smile$\hss}\kern\wd 0}

\def\Notind#1#2{#1\setbox 0=\hbox{$#1x$}\kern\wd 0\hbox to 0pt{\mathchardef
\nn=12854\hss$#1\nn $\kern 1.4\wd 0\hss}\hbox to
0pt{\hss$#1\mid$\hss}\lower.9\ht 0 \hbox to
0pt{\hss$#1\smile$\hss}\kern\wd 0}

\theoremstyle{definition}
\newtheorem{defn}{{Definition}}[section]
\newtheorem*{defn*}{{Definition}}

\newtheorem*{Q*}{{Question}}

\theoremstyle{plain}
\newtheorem{thm}[defn]{Theorem}
\newtheorem*{thm*}{Theorem}
\newtheorem{cor}[defn]{{Corollary}}
\newtheorem*{cor*}{Corollaire}

\newtheorem{prop}[defn]{{Proposition}}
\newtheorem*{prop*}{{Proposition}}
\newtheorem{lem}[defn]{{Lemma}}
\newtheorem{fact}[defn]{{Fact}}

\theoremstyle{remark}
\newtheorem{rmq}[defn]{{Remark}}

\newcommand{\ima}{\mathrm{Im}}
\renewcommand{\ker}{\mathrm{Ker}}

\title{Fields and rings with few types}
\author{C\'e{d}ric Milliet}
\address[]{Universit\'e de Lyon, Universit\'e Lyon 1\newline
Institut Camille Jordan, UMR 5208 CNRS\newline
43 boulevard du 11 novembre 1918\newline
69622 Villeurbanne Cedex, France}%
\address[Current address]{Université Galatasaray\newline
Facult\'e de Sciences et de Lettres\newline
D\'epartement de Math\'ematiques\newline
Çira\u gan Caddesi n°36\newline
34357 Ortaköy, Istamboul, Turquie\newline
}
\email[]{milliet@math.univ-lyon1.fr}%
\thanks{Most of this paper forms part of the author's doctoral dissertation, written in Lyon under the supervision of Wagner. Many thanks to Poizat and Point for rewarding comments and references.}

\begin{document}

\begin{abstract}Let $R$ be an associative ring with possible extra structure. $R$ is said to be \emph{weakly small} if there are countably many $1$-types over any finite subset of $R$. It is \emph{locally $P$} if the algebraic closure of any finite subset of $R$ has property $P$. It is shown here that a field extension of finite degree of a weakly small field either is a finite field or has no Artin-Schreier extension. A weakly small field of characteristic $2$ is finite or algebraically closed. Every weakly small division ring of positive characteristic is locally finite dimensional over its centre. The Jacobson radical of a weakly small ring is locally nilpotent. Every weakly small division ring is locally, modulo its Jacobson radical, isomorphic to a product of finitely many matrix rings over division rings.\end{abstract}

\maketitle

In~\cite{Mil1}, the author has begun the exploration of small and weakly small groups. He noticed that a weakly small group $G$ inherits locally several properties that omega-stable groups share globally. For instance $G$ satisfies local descending chain conditions. Every definable subset of $G$ has a local stabiliser with good local properties. If $G$ is infinite, it also possesses an infinite abelian subgroup, not necessarily definable though. Let's not forget that weakly small structures include omega-stable ones but also $\aleph_0$-categorical, minimal, and $d$-minimal ones. The following pages aim at classifying weakly small fields and rings, bearing in mind the classification of the particular cases cited above. The guiding line is that the formers should not differ much from the latters, at least \emph{locally}, in the following sense~: 

\begin{defn*}Let $M$ be a first order structure and $P$ any property. $M$ is said to be \emph{locally $P$} if every finitely generated algebraic closure in $M$ has property $P$.\end{defn*}

Let us bring to mind some known results about omega-stable, $\aleph_0$-categorical, and minimal rings. First concerning fields : an $\aleph_0$-categorical field is finite. Macintyre showed in 1971 that an omega-stable field is either finite or algebraically closed \cite{Macf}. Wagner drew the same conclusion for a small field, as well as for a minimal field of positive characteristic \cite{W1,WM}. Poizat extended the latter to $d$-minimal fields of positive characteristic \cite{dminimal}. Whether the same result holds even for a minimal field of characteristic zero is still unknown. We begin Section \ref{WSF} by giving another proof that a small field is either finite or algebraically closed, and derive that for an infinite weakly small field $F$, no field extension of $F$ of finite degree has Artin-Schreier extensions. It follows that a weakly small field of characteristic $2$ is finite or algebraically closed. Wagner had already noticed in \cite{W1} that a weakly small field is either finite or has no Artin-Schreier nor Kummer extensions. As a field extension of finite degree of a weakly small field has no obvious reason to be weakly small, our statement is a non-trivial improvement of \cite{W1}.

In section~\ref{WSDR}, we do not assume commutativity anymore and show that a weakly small division ring of positive characteristic is locally finite dimensional over its centre. Recall that superstable division rings \cite[Cherlin, Shelah]{CS} and even supersimple ones \cite[Pillay, Scanlon, Wagner]{PSW} are known to be fields.

We then turn to small difference fields. Hrushowski proved that in a superstable field, any definable field morphism is either trivial or has a finite set of fixed points \cite{Hrus}. We show that this also holds for a small field of positive characteristic.

We finish by rings in section~\ref{WSR}. The structure theory of an associative ring $R$ breaks usually into three parts : the study of its Jacobson radical ring $J$ ; the study of the quotient ring $R/J$, the so-called reduced ring ; and the gluing back process, buiding a ring with a given radical and reduced rings in the spirit of Wedderburn-Malcev's theorem \cite{}. An omega-stable ring $R$ is known to have a nilpotent Jacobson radical $J$ and $R/J$ is a finite product of matrix rings over finite or algebraically closed fields \cite[Cherlin, Reineke, Macintyre]{CR,Macf}. The Jacobson radical of an $\aleph_0$-categorical ring is nilpotent \cite[Cherlin]{Che,Che2}. For a weakly small ring $R$, we show that its radical is locally nilpotent (hence nil). Moreover, $R$ is locally modulo its radical the product of finitely many matrix rings over division rings.

\section{(Weakly) small tools}

Let us recall what a small theory and more generally a weakly small structure are.

\begin{defn}A \emph{theory} is \emph{small} if it has countably many complete $n$-types without parameters (or equivalently over any fixed finite parameter set) for every natural number $n$. A \emph{structure} is \emph{small} if so is its theory.\end{defn}

\begin{defn}[Belegradek]A structure is \emph{weakly small} if for any of its subsets $A$, there are countably many complete $1$-types over $A$.\end{defn}

For convenience of the reader, we state here the main results of \cite{Mil1} that will be needed in the sequel. We refer to the latter paper the reader willing to know more about weakly small groups.

In a weakly small structure $M$, for any finite parameter subset $A$ of $M$, the space $S_1(A)$ of complete $1$-types over $A$ is a countable compact Hausdorff space. It has an ordinal Cantor-Bendixson rank and one can compute the \emph{Cantor-Bendixson rank over $A$} of any of its element $p$. We write it $CB_A(p)$. For any $A$-definable set $X$ of arity~$1$, we write $CB_A(X)$ for the maximum Cantor rank of the complete $1$-types over $A$ countaining the formula defining $X$. Only a finite number of complete $1$-types over $A$ with same $CB_A$-rank as $X$ do contain the formula defining $X$. We call this natural number the \emph{Cantor-Bendixson degree of $X$ over $A$}, and write it $dCB_A(X)$.

What has been said for $1$-types of a weakly small structure is also valid for every $n$-type of a small structure over parameters in an arbitrary finite set. The following two Lemmas hold in any structure.


\begin{lem}\label{fibresf}Let $X$ and $Y$ be $A$-definable sets. Let $f$ be an $A$-definable map from $X$ onto $Y$. If the fibres of $f$ have no more than $n$ elements, then $f$ preserves the Cantor rank over $A$. Moreover, $$dCB_A(X)\leq dCB_A(Y)\leq n\cdot dCB_A(X)$$\end{lem}

\begin{lem}\label{CBacl}Let $X$ be a $\emptyset$-definable set, and $a$ an element algebraic over the empty set. Then $CB_a(X)$ equals $CB_\emptyset(X)$.\end{lem}

Lemma \ref{CBacl} allows to define the \emph{local Cantor rank} of an $a$-definable set $X$, to be its Cantor rank over any $b$ defining $X$ and having the same algebraic closure as $a$. We shall write $acl(B)$ for the algebraic closure of some set $B$, and $dcl(B)$ for its definable closure.

\begin{thm}[Weakly small descending chain condition]\label{dccm} In a weakly small group, the trace over $acl(\emptyset)$ of a descending chain of $acl(\emptyset)$-definable subgroups becomes stationary after finitely many steps.\end{thm}

For any $\emptyset$-definable set $X$ in a weakly small group $G$, if $\Gamma$ is the algebraic closure of a finite tuple $g$ from $G$, one can define the \emph{local almost stabiliser} of $X$ in $\Gamma$ to be $$Stab_\Gamma(X)=\{x\in \Gamma: CB_{x,g}(x X\Delta X)<CB_g(X)\}$$

$Stab_\Gamma(X)$ is a subgroup of $\Gamma$. If $\delta$ is any subgroup of $\Gamma$, we write $Stab_\delta(X)$ for $Stab_\Gamma\cap\delta$. Here is a local analogue of what happens for the stabiliser of a definable set of maximal Morley rank in an omega-stable structure~:

\begin{prop}\label{fini}Let $G$ be a weakly small group, $g$ a finite tuple of $G$, and $X$ a $g$-definable subset of $X$. If $\delta$ is a subgroup of $dcl(g)$ and if $X$ has maximal Cantor rank over $g$, then $Stab_\delta(X)$ has finite index in $\delta$.\end{prop}

Next proposition can be found in \cite{Mil}.

\begin{prop}\label{decom}Let $G$ be a small group, and $f$ a definable group homomorphism of $G$. There exists a natural number $n$ such that $\ker f^n\cdot\ima f^n$ equals $G$.\end{prop}

\section{Weakly small fields}\label{WSF}

We begin by proposing a lightened version of a result from Wagner.

\begin{lem}\label{perf}An infinite weakly small field is perfect.\end{lem}

\begin{proof}Let $f$ be a group homomorphism of some weakly small group $G$. Suppose that $f$ has finite kernel of cardinal $n$. An easy consequence of Lemma~\ref{fibresf} is that the image of $f$ has index at most $n$ in $G$. It follows that the image of the Froboenius map is a sub-field having finite additive (and multiplicative !) index.\end{proof}

\begin{thm}\label{con}A weakly small infinite field, possibly skew, has no definable additive nor multiplicative subgroup of finite index.\end{thm}

\begin{proof}Let $K$ be this field 
and let $H$ be a definable additive subgroup of $K$ having finite index. Suppose first that there is an infinite finitely generated algebraic closure $\Gamma$. Note that $\Gamma$ is a field. The intersection of $\lambda H\cap \Gamma$ where $\lambda$ runs over $\Gamma$ is a finite intersection by Theorem~\ref{dccm} hence has finite index in $\Gamma$. It is also an ideal of $\Gamma$ and must equal $\Gamma$. So $\Gamma$ is included in $H$. As this holds for any infinite finitely generated $\Gamma$, the group $H$ equals $K$.

Otherwise, $K$ is locally finite. By Wedderburn's theorem, $K$ is commutative and equals $acl(\emptyset)$. According to Theorem~\ref{dccm}, it satisfies the descending chain condition on definable subgroups. $K$ has a smallest definable additive subgroup of finite index, which must be an ideal, and hence equals $K$.

For the multiplicative case now, let us consider a multiplicative subgroup $M$ of $K^\times$ having index~$n$. Let us suppose first that there is $\delta$ an infinite finitely generated sub-field of $K$. There is no harm in extending $\delta$ so that each coset of $M$ be $\delta$-definable. $M$ has maximal Cantor rank over $\delta$ by Lemma~\ref{fibresf}, so its almost additive stabiliser in $\delta$ has finite additive index in $\delta$ by Proposition~\ref{fini}, as well as the almost stabiliser of any of its cosets. So the almost stabiliser of all the cosets is an ideal of $\delta$ having finite index, hence equals the whole of $\delta$. We finish as Poizat in \cite{dminimal}~: we have just shown that $1+aM\simeq aM$ for every coset $aM$, where $\simeq$ stands for equality up to small Cantor rank over $\delta$. For every coset $aM$, and every $x$ in $aM$ but a small ranked set, $1+x$ belongs to $aM$, so $x^{-1}+1\in M$, and the complement of $M$ has a small Cantor rank~: $M$ is exactly $K^\times$.

Otherwise $K$ is locally finite and has characteristic $p$. By Lemma~\ref{perf}, the group $K^\times$ is $p$-divisible, so $K^\times$ can not have a proper subgroup of finite index.
\end{proof}


Before going further, let us remind the reader with a few definitions. Let $L/K$ be a field extension. It is a \emph{Kummer extension} if it is generated by $K$ and one $n$th root of some element in $K$. It is an \emph{Artin-Schreier extension}, if $L$ is generated by $K$ and one $x$ such that $x^p-x$ belongs to $K$ ($x$ is called a \emph{pseudo-root} of $K$).

\begin{fact}[Artin-Schreier, Kummer \cite{Lan}]\label{AS} Let $K$ be a field of characteristic $p$ (possibly zero) and $L$ a cyclic Galois extension of finite degree $n$.\begin{itemize}
                \item[(i)] Suppose that $p$ is zero, or coprime with $n$. If $K$ contains $n$ distinct $n$th roots of~$1$, then $L/K$ is a Kummer extension.
		\item[(ii)] If $p$ equals $n$, then $L/K$ is an Artin-Schreier extension.\end{itemize}\end{fact}

\begin{cor}[Wagner \cite{W1}]\label{wa}A small field is finite or algebraically closed.\end{cor}

\begin{proof}Let $K$ be a small infinite field. For every natural number $n$, the $n$th power map has bounded fibres so its image has finite index by Lemma~\ref{fibresf}, and the map is onto by Theorem \ref{con}. So is the map mapping $x$ to $x^p-x$. 
We conclude as Macintyre in~\cite{Macf} for omega-stable fields~: first of all, $K$ contains every root of unity. For if $a$ is a $n$th root of $1$ not in $K$ with minimal order, $K(a)/K$ has degree $m<n$. By minimality of $n$, every $m$th root of unity is in $K$. If $m$ is zero or coprime with the characteristic of $K$, then $K(a)$ is of the form $K(b)$ with $b^m\in K$ after Fact~\ref{AS}. As $K^\times$ is divisible, $b$ is in~$K$, a contradiction. So $K$ has positive characteristic $p$, and $p$ divides $m$. $K(a)$ contains an extension of degree $p$ over $K$, which is an Artin-Schreier extension after Fact~\ref{AS}, a contradiction.

Secondly, if $K$ is not algebraically closed, it has a normal extension $L$ of finite degree~$n$, which must be separable as $K$ is perfect (Lemma~\ref{perf})~; its Galois group contains a cyclic sub-group of prime order $q$, the invariant field of which we call $M$. Note that $L$ is interpretable in a finite Cartesian power of $K$, so is small too. If $q$ is not the characteristic, as $K$ contains every $q$th root of~$1$, the extension $L/M$ is Kummer~; if $q$ equals the characteristic, $L/M$ is an Artin-Schreier extension, a contradiction in both cases.\end{proof}

Note that the first part of the previous proof still holds for weakly small fields~:

\begin{cor}\label{sol}An infinite weakly small field contains every root of unity, hence has no solvable by radical extension.\end{cor}

\begin{proof}If the extension $L/K$ is solvable by radical, there is a tower of fields $K=K_0\subset K_1\subset\dots\subset K_n=L$ so that each $K_{i+1}/K_i$ be generated by either a $n$th root or a pseudo-root of some element~$K_i$. But $K_0$ has no Artin-Schreier or Kummer extension by Corollary~\ref{wa}.\end{proof}


Note however that, as an algebraic extension of a weakly small field has no obvious reason to be weakly small, we cannot apply Macintyre's argument to deduce that a weakly small field is finite or algebraically closed. Nevertheless, stepping on the additive structure of the field, we can show that every algebraic extension of an infinite weakly small field is Artin-Schreier closed. This is a first step towards Problem 12.5 in \cite{W1} asking whether an infinite weakly small field is algebraically closed.

\begin{defn}An \emph{Abelian structure} is any abelian group together with predicates interpreting subgroups of its finite Cartesian powers.\end{defn}

As for a pure module, an Abelian structure has quantifier elimination up to positive prime formulas (see \cite[Weispfenning]{We}, or \cite[Theorem 4.2.8]{WStable})~:

\begin{fact}[Weispfenning]\label{pp} In an Abelian structure $A$, a definable set is a boolean combination of cosets of $acl(\emptyset)$-definable subgroups of Cartesian powers of~$A$.\end{fact}

An Abelian structure $A$ is stable~: let us take a formula $\varphi(x,y)$ such that $\varphi(x,0)$ defines an $acl(\emptyset)$-definable subgroup of $A^n$, and $\varphi(x,y)$ the coset of $y$. Two formulas $\varphi(x,a)$ and $\varphi(x,b)$ define cosets of the same group, so they must be equal or disjoint. It follows that there are countably many $\varphi$-types over any countable set of parameters. By Fact \ref{pp}, this is sufficient to show that $A$ is stable.

In a stable structure, we call a \emph{dense forking chain} any chain of complete types $p_q$ indexed by $\mathbf{Q}$ such that for every rational numbers $q<r$, the type $p_r$ be a forking extension of $p_q$.

Stable theories with no dense forking chains have been introduced in~\cite[Pillay]{Pil}. They generalise superstable ones. In a stable structure $M$ with no dense forking chains, every complete type (and not only $1$-types !) has an ordinal \emph{dimension}, and for any dimension $\alpha$, a \emph{Lascar $\alpha$-rank}. We shall write $dim(p)$ for the dimension of the type $p$, and $U_\alpha(p)$ for its $\alpha$-rank. They are defined as follows :

\begin{defn}[Pillay]For two complete types $p\subset q$, let us define the \emph{dimension of $p$ over $q$} written $dim(p/q)$ by the following induction.\begin{itemize}
\item $dim(p/q)$ is $-1$ if $q$ is a nonforking extension of $p$.
\item $dim(p/q)$ is at least $\alpha+1$ if there are non forking extensions $p'$ and $q'$ of $p$ and $q$, and infinitely many complete types $p_1,p_2,\dots$ such that $p'\subset p_1\subset p_2\subset\dots\subset q'$ and $dim(p_i/p_{i+1})\geq\alpha$ for all natural number $i$.
\item $dim(p/q)$ is at least $\lambda$ for a limit ordinal $\lambda$ if $dim(p/q)\geq\alpha$ for all $\alpha<\lambda$.
\end{itemize}\end{defn}

\begin{defn}[Pillay]For a complete type $p$, we set $dim(p)=dim(p/q)$ where $q$ is any algebraic extension of $p$.\end{defn}

\begin{defn}[Pillay]For every ordinal $\alpha$, we define inductively the $U_\alpha$-rank of a complete type $p$ by\begin{itemize}
\item $U_\alpha(p)$ is at least $0$.
\item $U_\alpha(p)$ is at least $\beta+1$ if there is an extension $q$ of $p$ such that $dim(p/q)\geq\alpha$ and $U_\alpha(q)\geq\beta$.
\item $U_\alpha(p)$ is at least $\lambda$ for a limit ordinal $\lambda$ if $U_\alpha(p)\geq\beta$ for all $\beta<\lambda$.
\end{itemize}
\end{defn}

To any type-definable stable group in $M$ can be associated the $U_\alpha$-rank and the dimension of any of its generic types over $M$. We refer the reader to~\cite[Herwig, Loveys, Pillay, Tanovi\'{c}, Wagner]{Pil,HLPTW} for more details. We shall just recall two facts : the Lascar inequalities which are still valid for the $U_\alpha$-rank, as well as their group version ; and the link between the $U_\alpha$-rank and the existence of a dense forking chain.

\begin{fact}[Lascar inequalities for $U_\alpha$-rank \cite{Pil,HLPTW}]\ 

\begin{enumerate}

\item In a stable structure, for every tuple $a,b$ and every set $A$, $$U_\alpha(b/Aa)+U_\alpha(a/A)\leq U_\alpha(ab/A)\leq U_\alpha(b/Aa)\oplus U_\alpha(a/A)$$\item For any type-definable group $G$ in a stable structure, and any type-definable subgroup $H$ of $G$, $$U_\alpha(H)+U_\alpha(G/H)\leq U_\alpha(G)\leq U_\alpha(H)\oplus U_\alpha(G/H)$$
\end{enumerate}\end{fact}

\begin{proof}We only prove point $(2)$, which does not appear anywhere to the author's knowledge but follows from $(1)$. Note that passing from the ambient structure $M$ to $M^{heq}$, one can use hyperimaginary parameters in $(1)$.

Let $tp(a/M)$ be a generic type of $G$. We write $a_H$ the hyperimaginary element which is the image of $a$ in $G/H$. The type $tp(a_H/M)$ is also a generic of $G/H$. Let $b$ be in the connected component $H^0$ of $H$, and generic over $M\cup\{a\}$. So $ab$ is a generic of $aH$ over $M\cup\{a\}$, hence over $M\cup\{a_H\}=M\cup\{(ab)_H\}$. As $a$ and $ab$ are in the same class modulo $G^0$, they realise the same generic type over $M$. It follows that $tp(a/M,a_H)$ is a generic of $aH$. Now apply point $(1)$ taking some generic of $G$ for $a$ and $b=a_H$.\end{proof}

See \cite[Lemma 7]{HLPTW} and \cite[Remark 9]{HLPTW} for

\begin{fact}\label{DFC}Let $p$ be a complete $n$-type. There is a dense forking chain of $n$-types containing $p$ if and only if the rank $U_\alpha(p)$ is not ordinal for every ordinal $\alpha$.\end{fact}

In a $\kappa$-saturated stable structure $M$, for any formula $\varphi(x,y)$, we can compute the Cantor-Bendixson of the topological space $S_\varphi(M)$ whose elements are the complete $\varphi$-types over $M$. Let $\psi(x)$ be another formula and $S_{\varphi,\psi}$ the subset of $S_\varphi(M)$ whose elements are consistent with $\psi$. It is a closed subset of $S_\varphi(M)$. The \emph{local $\varphi$-rank of $\psi$} is the Cantor-Bendixson rank of $S_{\varphi,\psi}$. We write it $CB_\varphi(\psi)$. The \emph{local $\varphi$-rank of a type $p$} is the minimum local $\varphi$-rank of the formulas implied by $p$. If $M$ is a stable group, the \emph{stratified $\varphi$-rank of $\psi$} is its $\phi$-rank, where $\phi(x,\overline{y})$ stands for the formula $\varphi(y_2\cdot x,y_1)$. We write it $CB^*_\varphi(\psi)$.

In a $\kappa$-saturated stable group $G$, let $H$ and $L$ be two type-definable subgroups. $H$ and $L$ are \emph{commensurable} if the index of their intersection is bounded (i.e. less that $\kappa$) in both of them. Recall that this is equivalent to $H$ and $L$ having the same stratified $\varphi$-rank for every formula $\varphi$.


\begin{thm}\label{stabm}Let be an Abelian structure with weakly small universe. Its theory has no dense forking chain.\end{thm}

\begin{rmq}Pillay showed that a small $1$-based structure has no dense forking chain \cite[Lemma 2.1]{Pil}. In particular, a small Abelian structure has no dense forking chain either. The difficulty of Theorem~\ref{stabm} comes from the fact that weak smallness does not bound a priori the number of pure $n$-types for $n\geq 2$, which is a crucial assumption in the proof of~\cite[Lemma 2.1]{Pil}.\end{rmq}

\begin{proof}According to Fact~\ref{DFC}, one just needs to show that for all finite tuple $\overline a$ and all set $A$, there is an ordinal $\alpha$ such that $U_\alpha(\overline a/A)$ is ordinal. Note that the first of Lascar inequalities for the $U_\alpha$-rank implies that $U_\alpha((a_1,\dots,a_n)/A)$ is less or equal to $U_\alpha((a_2,\dots,a_n)/Aa_1)\oplus U_\alpha(a_1/A)$. So, by induction on the arity of $\overline a$, and Fact~\ref{DFC} again, we may consider only $1$-types, and suppose for a contradiction that there be a dense forking chain of arity $1$.

\textit{(1) We first claim that there exists a dense ordered chain $(H_i)_{i\in\mathbf{Q}}$ of $acl(\emptyset)$-type-definable pairwise non commensurable subgroups.}

Let $(tp(a/A_i))_{i\in\mathbf{Q}}$ be a dense forking chain, that is $A_i$ is included in $A_j$ and $tp(a/A_j)$ forks over $A_i$ 
for all $i<j$. By Fact~\ref{pp}
, every formula appearing in $tp(a/A_i)$ is a boolean combination of cosets of $acl(\emptyset)$-definable groups. There is a smallest $acl(\emptyset)$-type-definable group $H_i$ such that the type $tp(a/A_i)$ contains the formulas defining $aH_i$. If $CB_\varphi^*(a/A_i)<CB_\varphi^*(aH_i)$ for some formula $\varphi$, then there is an $acl(\emptyset)$-definable subgroup $G_i$ with the formula defining $aG_i$ included in $tp(a/A_i)$, and $CB_\varphi^*(aG_i)<CB_\varphi^*(aH_i)$. This implies $CB_\varphi^*(G_i)<CB_\varphi^*(H_i)$ and contradicts the minimality of $H_i$. It follows that $tp(a/A_i)$ is a generic type of $aH_i$. Moreover, $aH_i$ is $A_i$-type-definable. For all $i<j$, the type $tp(a/A_j)$ forks over $A_i$ so there must be a formula $\varphi$ such that $aH_j$ and $aH_i$ have different stratified $\varphi$-ranks. Then, one has $CB_\varphi^*(H_i)<CB_\varphi^*(H_j)$ so $H_i$ and $H_j$ are non-commensurable groups.

\textit{(2) Let us now build $2^{\aleph_0}$ complete $1$-types over $\emptyset$.}

As the structure is stable, each $H_i$ is the intersection of $acl(\emptyset)$-definable groups $H_{ij}$. Let $\overline{H_{ij}}$ stand for the $\emptyset$-definable union of the conjugates of $H_{ij}$ under $Aut(\emptyset)$. Let us call $\widehat{H_i}$ the $\emptyset$-type-definable intersection of the $\overline{H_{ij}}$ over $j$. For every real number~$r$, we call $p'_r$ the partial type defining $\bigcap_{i\geq r} \widehat{H_i}$ and $p_r$ the following partial type $$p'_r\cup \{\psi:\psi\text{ formula over }\emptyset\text{ with }CB^*_\varphi(\neg\psi)<CB^*_\varphi(p'_r)\text{ for some }\varphi\}$$ Note that every formula $\psi$ in the second part of the type $p_r$ above is contained in every generic type of the structure (and also in the generic types of $\bigcap_{i\geq r} \widehat{H_i}$). Every $p_r$ is thus consistent. We claim that if $r\neq q$, then $p_r$ and $p_q$ are inconsistent. In any stable group, if $G_1,G_2,\dots$ are decreasing definable subgroups, for every formula $\varphi$, there is an index $i_\varphi$ such that the equality $CB^*_\varphi(\bigcap_{i\geq 1} {G_i})=CB^*_\varphi(G_j)$ holds for all $j>i_{\varphi}$. So one can find two indexes $i$ and $j$ (depending on $\varphi$) such that all the following equalities hold $$CB^*_\varphi(\bigcap_{i\geq r} \widehat{H_i})=CB^*_\varphi(\widehat{H_i})=CB^*_\varphi( \overline{H_{ij}})$$ $$CB^*_\varphi(\bigcap_{i\geq r} {H_i})=CB^*_\varphi({H_{i}})=CB^*_\varphi({H_{ij}})$$ As the stratified $\varphi$-ranks are preserved under automorphisms, and as the rank of a finite union equals the maximum of the ranks, we get $$CB^*_\varphi(\bigcap_{i\geq r} \widehat{H_i})=CB^*_\varphi(\bigcap_{i\geq r} {H_i}) $$ By point $(1)$, the groups $H_i$ are pairwise non-commensurable. It follows from the last equality that every pair of elements of the chain $(\widehat{H_i})_{i\in\mathbf{Q}}$ has at least one $CB^*_\varphi$-rank distinguishing them. So the types $(p_r)_{r\in\mathbf{R}}$ are pairwise inconsistent. We may complete them and build $2^{\aleph_0}$ complete $1$-types.\end{proof}

Note that if $G$ is a stable group with no dense forking chain, and $f$ a definable group morphism from $G$ to $G$ with finite kernel, then $G$ and $f(G)$ have same dimension and $U_\alpha$-rank. This follows from the second Lascar inequality applied to $G$ and $\ker G$. More generally~:

\begin{lem}\label{ff}Let $X$ and $Y$ be type-definable sets, and let $f$ be a definable map from $X$ onto $Y$ the fibres of which have no more than $n$ elements for some natural number~$n$. Then $X$ and $Y$ have the same dimension, and same $U_\alpha$-rank for every ordinal~$\alpha$.\end{lem}

\begin{proof}One just needs to notice that if $q$ is some type in $X$, and if $p$ is an extension of $q$, then $p$ is a forking extension of $q$ if and only if $f(p)$ is a forking extension of $f(q)$.\end{proof}

\begin{prop}\label{fff}Let $G$ be a stable group without dense forking chain, and let $f$ be a definable group morphism from $G$ to $G$. If $f$ has finite kernel, its image has finite index in $G$.\end{prop}

\begin{proof}Let us write $H$ for $f(G)$, and let us apply the first Lascar equality. We get $U_\alpha(H)+U_\alpha(G/H)\leq U_\alpha(G)$. But $H$ and $G$ have the same $U_\alpha$-rank after Lemma~\ref{ff}. It follows that $U_\alpha(G/H)$ is zero. This holds for every ordinal $\alpha$, so $dim(G/H)$ is $-1$. This means that $G/H$ is finite.

\end{proof}

\begin{rmq}In Proposition~\ref{fff}, one cannot bound the index of the image of $f$ with the cardinal of its kernel. Consider for instance the superstable group $(\mathbf{Z},+)$, and the maps $f_n$ mapping $x$ to the $n$ times sum $x+\dots+x$, when $n$ ranges among natural numbers.\end{rmq}

\begin{cor}Every algebraic extension of an infinite weakly small field is Artin-Schreier closed.\end{cor}

\begin{proof}Let $K$ be an infinite weakly small field of positive characteristic $p$, let $L$ be an algebraic extension of $K$ and $f$ the Artin-Schreier map from $L$ to $L$. We consider the additive structure of $L$, together with $f$~: it is an Abelian structure with weakly small universe $K$. It has no dense forking chain by Theorem~\ref{stabm}. The map $f$ has finite fibres so $f(L)$ has finite additive index in $L$ by Proposition~\ref{fff}. But $K$ has no proper definable additive subgroup of finite index by Theorem~\ref{con}, so neither has any finite Cartesian power of $K$, thus $f$ is onto.\end{proof}

\begin{cor}\label{degp}The degree of an algebraic extension of an infinite weakly small field of positive characteristic $p$ is not divisible by $p$.\end{cor}

\begin{proof}Let $K$ be this infinite weakly small field~; it is perfect by Theorem~\ref{con}. If there is an algebraic extension the degree of which is divisible by $p$, there is also a normal separable extension $L$ of finite degree divisible by $p$. Its Galois group has a subgroup of order $p$, the invariant field of which we note $K_1$. The extension $L/K_1$ is an Artin-Schreier extension, a contradiction.\end{proof}

\begin{cor}\label{WS2}A weakly small field of characteristic two is either finite or algebraically closed.\end{cor}

\begin{proof}If it is infinite and not algebraically closed, it has a normal separable algebraic extension of finite degree. According to Corollaries~\ref{degp} and~\ref{sol}, its Galois group neither has even order, nor is soluble, a contradiction to Feit-Thomson's Theorem.\end{proof}


\section{Weakly small division rings}\label{WSDR}

Recall that a superstable division ring is a field \cite[Cherlin, Shelah]{CS}. The author has shown that a small division ring of positive characteristic is a field \cite{Mil}. It is still unknown whether this extends to weakly small division rings. In this section we show, at least, that every finitely generated algebraic closure in a weakly small division ring has finite dimension over its centre. With the previous section, this implies that a weakly small division ring of characteristic $2$ is a field.

From now on, let $D$ be an infinite weakly small division ring. 
If $K$ is a definable sub-division ring of $D$, one may view $D$ as a left or right vector space over $K$. However, we will not distinguish between the left and right $K$-dimension of $D$ thanks to~:

\begin{lem}If $K$ is a definable sub-division ring of $D$, and if $D$ has finite left or right dimension over $K$, then $D$ has finite right and left dimension over $K$, and those dimensions are the same. Moreover, there is a set which is both a left and right $K$-basis of $D$.\end{lem}

\begin{proof}Let $f_1,\dots,f_n$ be a left and right $K$-free family from $D$, with $n$ maximal. Let $F_r$ and $F_l$ be the set of respectively right and left linear $K$-combinations of the $f_i$. If $F_r< D$ and $F_l<D$, then $F_r\cup F_l< D$, a contradiction with $n$ being maximal. So suppose that $D$ equal $F_r$. The group homomorphism from $D^+$ mapping a right decomposition $\sum_{i=1}^{n} f_ik_i$ to $\sum_{i=1}^{n} k_if_i$ is a definable embedding, hence surjective after Lemma~\ref{fibresf}. Thus $F_l$, $F_r$ and $D$ are equal.\end{proof}




\begin{prop}\label{infi}The centre of an infinite weakly small division ring is infinite.\end{prop}

\begin{proof}We may assume that $D$ has non-zero characteristic, as this obviously holds in zero characteristic. We may also assume that $D$ is not locally finite and has an element $b$ of infinite order. It follows from Corollary~\ref{sol} that $Z(C(b))$ contains every root of $1$. We claim that all those roots are in $Z(D)$. Suppose not, and let $a$ be non central with $a^q=1$ and $q$ a prime number. According to a lemma of Herstein~\cite[Lemma 3.1.1]{Her}, there exists a natural number $n$ and an $x$ in $D$ with $xax^{-1}=a^n$ but $a^n\neq a$ . If $x$ has finite order, the division ring generated by $x$ and $a$ is finite, a contradiction to Wedderburn's Theorem. So $x$ has infinite order. Conjugating $q-1$ times by $x$, we get $x^{q-1} ax^{-{q+1}}=a^{n^{q-1}}=a$. Note that $x^{q-1}$ has infinite order, so $Z(C(x^{q-1}))$ contains $x$ by Corollary~\ref{sol}. It follows that $a$ and $x$ commute, a contradiction.\end{proof}

\begin{cor}\label{puiss}An element and a power of it have the same centraliser.\end{cor}

\begin{proof}Let $a$ be in $D$. We obviously have $C(a)\leq C(a^n)$. Conversely, by Proposition~\ref{infi}, the field $Z(C(a^n))$ is infinite. Corollary~\ref{sol} implies that it contains $a$.\end{proof}

\begin{rmq}It follows that every element having finite order lies in the centre. Similarly, for every non constant polynomial $P$ with coefficients in the centre having a soluble Galois group, $P(a)$ and $a$ have the same centraliser in $D$. If $D$ is in addition small, this holds for every non constant polynomial with coefficients in the centre.\end{rmq}




\begin{cor}\label{Cent}Let $\overline{a}$ be some finite tuple in $D$. The sets $acl(\overline{a})$, $dcl(\overline{a})$ and $\overline{a}$ have the same centraliser in $D$.\end{cor}

\begin{proof}The inclusions $C(acl(\overline{a}))\leq C(dcl(\overline{a}))\leq C(\overline{a})$ are easy. Conversely, suppose $x$ commutes with $\overline{a}$ and let $y$ be in $acl(\overline{a})$. For every natural number $m$, the elements $y^{x^m}$ and $y$ are conjugated by the action of the automorphisms group fixing $\overline{a}$ pointwise. So there must be two distinct natural numbers $n$ and $m$ so that $y^{x^n}$ and $y^{x^m}$ be equal~: $y$ commutes with a power of $x$, hence with $x$ by Corollary~\ref{puiss}.
\end{proof}

\begin{lem}Let $\gamma$ stand for the conjugation map by some $a$ in $D$. For all $\lambda$ in $D$, the kernel of $\gamma-\lambda.id$ is a $C(a)$-vector space having dimension at most $1$.\end{lem}

\begin{proof}Let some non zero $x$ and $y$ be in the kernel of $\gamma-\lambda.id$. The equalities $x^a=\lambda x$ and $y^a=\lambda y$ yield $(y^{-1}x)^{a}=y^{-1}x$.\end{proof}

\begin{lem}\label{dimf}In a weakly small division ring of positive characteristic, for all $a$, every finitely generated algebraic closure $ \Gamma$ containing $a$ is a finite dimensional $C_{\Gamma}(a)$-vector space.\end{lem}

\begin{proof}We write $f$ for the endomorphism mapping $x$ to $x^a - x$. Let $K$ and $H$ stand for the kernel and the image of $f$ respectively. Note that $f$ is not onto, as otherwise there would be some $x$ verifying $x^a=x+1$ and $x^{a^p}=x+p=x$, a contradiction to Lemma~\ref{puiss}. Let $\tilde{f}$ be the restriction of $f$ from $D^+/K$ to $D^+/K$. The set $H$ is a $K$-vector space so the intersection $H\cap K$ is an ideal of $K$, which must be trivial $f$ is not onto. The map $\tilde{f}$ is injective hence surjective, so we get $D=H\oplus K$. This yields $$\Gamma=H\cap \Gamma\oplus K\cap \Gamma$$ The intersection $I$ of the sets $\lambda H\cap  \Gamma$ where $\lambda$ runs over $\Gamma$ is a finite intersection, of size $n$ say~: it is a left ideal of $\Gamma$, hence zero. But $H\cap \Gamma$ is a $K\cap \Gamma$-vector space having codimension $1$, so $I$ has codimension at most $n$.\end{proof}



\begin{thm}A weakly small division ring of positive characteristic is locally finite dimensional over its centre.\end{thm}

\begin{proof}Let $\Gamma$ be finitely generated algebraic closure, and $D_0,\dots D_{n+1}$ a maximal chain of centralisers of elements in $\Gamma$ such that the chain $$ \Gamma>D_1\cap \Gamma>\cdots>D_n\cap \Gamma> D_{n+1}\cap \Gamma$$ be properly descending, and $D_{n}\cap \Gamma$ be minimal non commutative. The fields extensions $D_i\cap \Gamma/D_{i+1}\cap \Gamma$ are finite by Lemma~\ref{dimf}. As $D_{n+1}\cap \Gamma$ is a field, $\Gamma$ has finite dimension over its centre, bounded by $[\Gamma:D_{n+1} \cap \Gamma]^2$ according to \cite[Corollary 2 p.49]{Co}.\end{proof}

\begin{cor}\label{com}A small division ring of positive characteristic is a field.\end{cor}

\begin{proof}Let $\Gamma$ be the algebraic closure of a finite tuple $\overline a$. By Corollary~\ref{Cent}, we have $$Z(\Gamma)=Z(C(\Gamma))\cap \Gamma=Z(C(\overline{a}))\cap\Gamma$$  By \cite{W1}, $Z(C(\overline{a}))$ is algebraically closed. It follows that $Z(\Gamma)$ is relatively algebraically closed in $\Gamma$, so a small division ring is locally commutative, hence commutative.\end{proof}

\begin{cor}A weakly small division ring of characteristic $2$ is a field.\end{cor}

\begin{proof}Follows from Corollary~\ref{WS2} with the same proof as Corollary~\ref{com}.\end{proof}

\begin{cor}Vaught's conjecture holds for the pure theory of a positive characteristic division ring.\end{cor}

\begin{proof}If the theory of an infinite pure division ring has fewer than $2^{\aleph_0}$ denumerable models, it is small~: it is the theory of a algebraically closed field, which has countably many denumerable models as noticed in \cite{W1}.\end{proof}

In positive characteristic, we can just say the following~:

\begin{prop}If $D$ is small, let $a$ be outside the centre, and write $\gamma$ for the conjugation by $a$. For all non-zero polynomial $a_nX^n+\dots+a_1X+a_0$ with coefficients in the centre of $D$, the morphism $a_n\gamma^n+\dots+a_1\gamma+a_0Id$ is onto.\end{prop}

\begin{proof}Let $K$ be the field $C(a)$. As $Z(D)$ is algebraically closed, $P$ splits over $Z(D)$. As a product of surjective morphisms is still surjective, it suffices to show the result for some irreducible $P$. Let $\lambda$ be in the centre, let $f$ be the morphism $\gamma-\lambda.id$, and let $t$ be outside the image of $f$. The map $f$ is a $K$-linear map~; its kernel must be a line or a point. According to Proposition~\ref{decom}, we get $D=\ker f^m+\ima f^m$ for some natural number $n$. Set $H$ the image of $f^m$, and $L$ its kernel. Note that $L$ has finite $K$-dimension. We may replace $L$ by a definable summand of $H$, and assume that $L$ and $H$ be disjoints. Let $\Gamma$ an infinite finitely generated algebraic closure containing $t$, $a$, some $b$ which does not commute with $a$, and the $K$-basis of $L$. We still have $$\Gamma=L\cap \Gamma\oplus H\cap \Gamma$$
The intersection $I$ of the sets $\lambda H\cap \Gamma$, where $\lambda$ runs over $\Gamma$ is a finite intersection by Theorem~\ref{dccm}~: it is an ideal of $\Gamma$ which does not contain $t$, hence zero. But $H\cap \Gamma$ has finite $K\cap \Gamma$-codimension, hence so has $I$. According to \cite[Corollary 2 p.49]{Co}, we have $$[\Gamma:K\cap \Gamma]=[\Gamma:C_\Gamma(a)]=[Z(\Gamma)(a):Z(\Gamma)]<\infty$$ But $Z(\Gamma)$ is nothing more than $Z(C_D(\Gamma))\cap \Gamma$. By Corollary~\ref{Cent}, $Z(C_D(\Gamma))$ is an algebraically closed field so $a$ belongs to $Z(\Gamma)$, a contradiction.\end{proof}

\begin{cor}In a small field, the conjugation by any element generates a central division algebra.\end{cor}

\section{Small difference fields}

It is a common phenomenon to weakly small and superstable groups, and to stable groups without dense forking chain, that when a definable group homomorphism has a somehow small kernel, its image has to be somehow big. See \cite[Proposition 1.7]{Poi2}, or \cite[Corollary 6]{W2}. As for a definable endomorphism of a small field, we have the following~:

\begin{prop}\label{endosurj}Let $K$ be a small infinite field, $F$ a definable subfield, and $f$ a non trivial $F$-linear endomorphism of $K$, the kernel of which has finite $F$-dimension. Then $f$ is onto.\end{prop}

\begin{proof}
By Proposition~\ref{decom}, the equality $K=\ker f^m+\ima f^m$ holds for some natural number $m$. Let $H$ be the image of $f^m$ and $L$ its kernel. Note that if $n$ is the dimension of $\ker f$ over $F$, then $\ker f^m$ has dimension at most $nm$ over $F$. We may replace $L$ by a definable supplement of $H$ in $K$, an suppose that $L$ and $H$ are disjoint. If $F$ is finite, so is $L$. It follows that $H$ equals $K$, and $f$ is onto. So let us suppose $F$ infinite. Let $\Gamma$ be a finitely generated algebraic closure containing an $F$-basis $\overline{b}$ of $L$. As $L$ and $H$ are disjoint, we get $$\Gamma=L\cap \Gamma\oplus H\cap \Gamma$$ where $L\cap \Gamma$ is a finite dimensional $F\cap\Gamma$-vector space. The intersection $I$ of the sets $\lambda H\cap \Gamma$, where $\lambda$ runs over $\Gamma$ is a finite intersection. It is an ideal of $\Gamma$. Note that this holds for every finitely generated algebraic closure $\Gamma$ containing $\overline{b}$. If $I=\Gamma$ for every such $\Gamma$, then $f$ is surjective. So we may assume that $I$ is zero for every sufficiently large $\Gamma$. But $H\cap \Gamma$ has finite $F\cap \Gamma$-codimension, and so has $I$. It follows that $\Gamma$ is an algebraic extension of $\Gamma\cap F$. But $\Gamma\cap F$ is algebraically closed as $\Gamma$ and $F$ are. Hence $F$ contains $\Gamma$. As this holds for every large enough $\Gamma$, the field $F$ and $K$ are equal. If $\ker f$ has $F$-dimension $1$, then $f$ is trivial, a contradiction. So $\ker f$ is zero, and $f$ is onto by Lemma~\ref{fibresf}.\end{proof}


Let $K$ be a small infinite field, together with a definable field morphism $\sigma$. We call $F$ the subset of points in $K$ fixed by $\sigma$. It is a definable subfield of $K$ hence either finite or algebraically closed. The kernel of $\sigma$ being an ideal of $K$, either $\sigma$ is zero, or it is injective hence surjective by Lemma~\ref{fibresf}.

\begin{lem}\label{truc}For all polynomial $P$ with coefficients in $K$ and degree $n$, the kernel of $P(\sigma)$ is an $F$-vector space having dimension at most $n$.\end{lem}

\begin{proof}Let $x_0,x_1,\dots,x_n$ be solutions of the equation$$\sigma^n(x)+\displaystyle\sum_{i=0}^{n-1}a_i\sigma^i(x)=0$$ and let $C(x_0,x_1,\dots,x_n)$ be their Casoratian, defined by

\begin{eqnarray*}
C(x_0,x_1,\dots,x_n) 	&= & 
\left|
\begin{array}{cccc}
x_0 & x_1 & \cdots & x_n\\
\sigma(x_0) & \sigma(x_1) & \cdots & \sigma(x_n)\\
\vdots & \vdots & \ddots & \vdots\\
\sigma^{n}(x_0) & \sigma^{n}(x_1) & \cdots & \sigma^{n}(x_n)
\end{array}
\right| \\
						& = & \left|
\begin{array}{cccc}
x_0 & x_1 & \cdots & x_n\\
\sigma(x_0) & \sigma(x_1) & \cdots & \sigma(x_n)\\
\vdots & \vdots & \ddots & \vdots\\
-\displaystyle\sum_{i=0}^{n-1}a_i\sigma^i(x_0) & -\displaystyle\sum_{i=0}^{n-1}a_i\sigma^i(x_1) & \cdots & -\displaystyle\sum_{i=0}^{n-1}a_i\sigma^i(x_n)
\end{array}
\right|\\
						& = & 0
\end{eqnarray*}So the solutions $x_0,x_1,\dots,x_n$ are linearly dependent over $F$ according to \cite[Lemma II p. 271]{Cohn}.\end{proof}

\begin{lem}\label{fix} If the set of points fixed by $\sigma$ is infinite, the set of points fixed by $\sigma$ and $\sigma^n$ are the same for every natural number~$n$.\end{lem}

\begin{proof}The field $F$ is algebraically closed. If $\sigma^n$ fixes $x$, $\sigma$ fixes the symmetric functions of the roots $x,\sigma(x),\dots,\sigma^{n-1}(x)$, hence the roots.\end{proof}

\begin{thm}In a small field of positive characteristic, the only definable field morphism the set of points fixed by which is infinite, is the identity.\end{thm}

\begin{proof}Otherwise, the map $\sigma-Id$ is onto after Proposition~\ref{endosurj} and Lemma~\ref{truc} so there is some $x$ satisfying $\sigma(x)=x+1$, hence $\sigma^p(x)=x+p=x$, a contradiction with Lemma~\ref{fix}.\end{proof}



\section{Weakly small rings}\label{WSR}

All the rings considered here are associative. They may neither have a unit nor be abelian. Let $R$ be a ring. An element $r$ is said to be \emph{nilpotent of nilexponent $n$} if $n$ is the least natural number with $r^n=0$. $R$ is \emph{nil} if there is some natural number $n$ such that every element is nil of nilexponent at most $n$. Its \emph{nilexponent} is the least such $n$. We write $R^n$ for the $n$ times Cartesian product of $R$, and $R^{(n)}$ for the set $\{r_1\cdots r_n:(r_1,\dots,r_n)\in R^n\}$. The ring $R$ is \emph{nilpotent} if $R^{(n)}$ is zero for some $n$. Its \emph{nilpotency class} is the least such $n$. An \emph{idempotent} is any non-zero element $e$ with $e^2=e$. Two idempotents $e,f$ are \emph{orthogonal} if $ef=fe=0$.

Let $A$ be a subset of $R$. An element $r$ \emph{left annihilates $A$} if $rA$ is zero. We write $Ann_R(A)$ for the \emph{left annihilators of $A$ in $R$}, that is the set of elements in $R$ that left annihilates $A$. Symmetrically, $Ann^R(A)$ will stand for the right annihilator of $A$. Note that $Ann_RAnn^RAnn_R(A)$ equals $Ann_R(A)$.

The \emph{characteristic of $R$} is the least non-zero natural number $n$ such that the $n$ times sum $r+\dots+r$ is zero for every $r$ in $R$. If such a number does not exists $R$ has characteristic zero.

We begin by "dimension $1$" rings in the sense of \cite{dminimal}. Recall that a $d$-minimal group is abelian-by-finite \cite{dminimal}.

\begin{prop}\label{minring}An infinite ring with no definable infinite proper subgroup is a ring with trivial multiplication, or a division ring.\end{prop}

\begin{proof}Let $R$ be such a ring. Any definable group morphism of $R$ is either zero or onto. Let us suppose the multiplication non trivial. There exists some $r$ so that $rR$ equal $R$. There is some $e$ such that $re$ equal $a$. So $eR$ equals $R$. If there exists some $x$ with $ex-x=r$, then $r^2$ is zero, and $R$ is zero. Thus $ex-x$ must be zero for all $x$, and $e$ is a left unity. Symmetrically, $R$ has a right unit, which must be $e$. If $sR$ is zero, then $s$ is zero, so the multiplication by a non-zero element is onto, and $R$ is a division ring.\end{proof}

\begin{cor}A $d$-minimal ring has an ideal of finite index which is a field or a ring with trivial multiplication.\end{cor}

\begin{proof}The ring has a smallest definable additive subgroup of finite index $I$ which is an ideal, with no proper infinite subgroup. If the multiplication is non-trivial, $I$ is a division ring by Proposition~\ref{minring}. $I$ has a smallest multiplicative subroup $G$ of finite index, which is abelian by~\cite{dminimal}. Its centraliser $C_I(G)$ is an infinite division ring, and equal $I$. So the centre of $I$ is infinite, and equals $I$.\end{proof}

\subsection{General facts about weakly small rings}

\begin{lem}Let $R$ be a weakly small ring.\begin{enumerate}
\item If some element does not left divide zero, $R$ has a left unit.
\item If $R$ is unitary, an element is left invertible if and only if it is right divisible, and its right and left inverses are the same.
\item An element is a left zero divisor if and only if it is not left invertible.\end{enumerate}\end{lem}

\begin{proof}$(1)$ Let $r$ be a non left zero divisor. Left multiplication by $r$ is injective, so surjective by Lemma~\ref{fibresf}, and there is some $e$ such that $re=r$. For all $s$, $r(es-s)$ is zero so $es$ equals $s$, and $e$ is a left unit. $(2)$ If $rs=e$ holds, then right multiplication by $r$ is injective, hence surjective~; there exists some $t$ so that $tr=e$. Hence $te=trs=es$. $(3)$ If $r$ does not left divide zero, then $R$ has a left unit, and left multiplication by $r$ is onto~: the pre-image of the left unit is a right inverse of $r$.\end{proof}

\begin{cor}A weakly small ring with no zero divisor is a division ring.\end{cor}

Let us state two chain conditions on ascending chains of annihilators~:

\begin{prop}\label{Ch1}In a weakly small ring $R$, there is no properly ascending chain of left annihilators of the kind $Ann_R(\delta_1)\leq Ann_R(\delta_2)\leq\cdots\leq Ann_R(\delta_i)\leq\cdots$, where the sets $\delta_i$ lie in some finitely generated definable closure $\delta$.\end{prop}

\begin{proof}The chain $Ann^RAnn_R(\delta_1)\geq\cdots\geq Ann^RAnn_R(\delta_{i})$ is decreasing. Let $n$ be some natural number so that the Cantor rank and degree over $\delta$ of $Ann^RAnn_R(\delta_{n})$ be minimal. As $\delta_n+Ann^RAnn_R(\delta_{n+1})$ is included in $Ann^RAnn_R(\delta_n)$, the set $\delta_n$ is included in $Ann^RAnn_R(\delta_{n+1})$, so $Ann_R(\delta_{n+1})\leq Ann_R(\delta_n)$.\end{proof}


\begin{prop}\label{Ch2}In a weakly small ring, there is no properly ascending chain of annihilators $Ann_\Gamma(\Gamma_1)\leq Ann_\Gamma(\Gamma_2)\leq\cdots\leq Ann_\Gamma(\Gamma_{i})\leq\cdots$, where the sets $\Gamma_i$ lie in some finitely generated algebraic closure $\Gamma$.\end{prop}

\begin{proof}For every set $X$, the set $Ann X$ is type-definable with parameters in $X$. So $Ann^\Gamma Ann_\Gamma(\Gamma_i)$ is $\Gamma$-type-definable. The chain $Ann^\Gamma Ann_\Gamma(\Gamma_i)$ is descending, so $Ann^\Gamma Ann_\Gamma(\Gamma_n)$ equals $Ann^\Gamma Ann_\Gamma(\Gamma_{n+1})$ for some natural number~$n$ after the weakly small chain condition.\end{proof}

\begin{rmq}Propositions~\ref{Ch1} and~\ref{Ch2} are incomparable. The first one is global, with parameters in a definable closure, whereas the second one is local, but with parameters in an algebraic closure. They both hold for chains of right annihilators.\end{rmq}

\begin{cor}\label{rdecom}In a weakly small ring $R$, for every element $r$, there is a natural number $n$ such that$$R=r^n.R\oplus Ann_R(r^n)$$\end{cor}

\begin{proof}By Proposition~\ref{Ch1}, the chain $Ann_R(r),Ann_R(r^2),\dots$ becomes stationary at some step $n$. It follows that the right multiplication map by $r^n$ is an injective homeomorphism of the group $R^+/Ann_R(r^n)$. By Lemma~\ref{fibresf}, it must be onto.\end{proof}





\subsection{The Jacobson radical}

Every abelian group can be given a ring structure with trivial multiplication. Given any ring, one may be willing to isolate its "trivial" part. This is one reason to introduce the Jacobson radical. Among other notions of radical, the one introduced by Jacobson seems to be the more efficient to establish structure theorems for rings with zero radicals~:

\begin{fact}\label{WA}\emph{(Wedderburn-Artin \cite[Theorem 2.1.7]{Her})} A right Artinian ring with zero Jacobson radical is isomorphic to a finite Cartesian product of matrix rings over fields.\end{fact}

Recall that a ring is \emph{right Artinian} if every decreasing chain of right ideals is stationary. An element $r$ is \emph{right quasi-regular} if there is some $s$ such that $r+s+rs$ is zero. We write $J(R)$ the Jacobson radical of $R$, defined for our purpose by~:

\begin{fact}\label{Jacob}\emph{(Jacobson \cite[Theorem 1.2.3]{Her})} The \emph{Jacobson radical} of a ring is the unique maximal right ideal the elements of which are right quasi-regular.\end{fact}

Whereas most other radicals are definable in second order logic, the Jacobson radical is definable by the following first order formula~: $\forall y\exists z(xy+z+xyz=0)$.

\begin{prop}\label{rad}In the Jacobson radical of a weakly small ring, the algebraic closure of any finite tuple is nilpotent.\end{prop}

\begin{proof}Let $J$ be the Jacobson radical, and $\Gamma$ be a finitely generated algebraic closure in $J$. By Proposition~\ref{Ch2}, there is a natural number $n$ such that $Ann_\Gamma(\Gamma^{(n)})$ equals $Ann_\Gamma(\Gamma^{(n+1)})$. If $\Gamma$ is not nilpotent, then $\Gamma\setminus Ann_\Gamma(\Gamma^{(n)})$ is not empty. Among the $a$ in $\Gamma\setminus Ann_\Gamma(\Gamma^{(n)})$, let us choose one such that the group $aJ\cap \Gamma$ be minimal (among the groups $\{\gamma J\cap\Gamma:\gamma\in\Gamma\}$). This is possible by the weakly small chain condition~\ref{dccm}. Neither $a\Gamma^{(n)}$, nor $a\Gamma^{(n+1)}$ equal zero, so there exists some $b$ in $\Gamma$ such that $ab\Gamma^{(n)}$ is not zero. We claim that $ab$ does not belong to $abJ$. Otherwise, there is some $c$ in $J$ with $ab=abc$. So $-c$ belongs to $J$ too, and there exists some $d$ such that $d-c-cd$ equals zero. It follows that the equality $abd=abcd=ab(d-c)$ hold, hence $abc$ is zero. So $ab$ is zero, a contradiction. Thus, one has $abJ\cap \Gamma<aJ\cap \Gamma$, a contradiction.\end{proof}



\begin{rmq}Let $J$ be the radical of a small ring. By Proposition~\ref{rad}, $J$ is nil. A compactness argument ensures that its nil exponent must be bounded. Call $n$ the nilexponent of $J$. Dubnov-Ivanov-Nagata-Higman's Theorem states that \emph{"a nil algebra of exponent $n$ over a field of characteristic either zero or a prime number greater than $n$ is nilpotent of class at most $2^n-1$"}. See \cite[Jacobson il me semble]{Jac} for a proof of that. The proof extends naturally to the ring context : " ". By \cite{WQ}, the additive group of $J$ is the direct sum of two subgroups $J_0$ and $J_m$, where $J_0$ is additively divisible and $J_m$ has additive exponent $m$. One can easily see that $J_0$ and $J_m$ are ideals, of characteristic zero and $m$ respectively. After Dubnov-Ivanov-Nagata-Higman's Theorem, $J_0$ is nilpotent. Should $m$ be a prime number greater than $n$, $J_m$ (hence $J$) would be nilpotent too. Is the radical of a small ring nilpotent? 
\end{rmq}

\subsection{Abelian ring with zero radical}

Let $R$ be a weakly small abelian ring with zero radical. Note that $R$ has no nilpotent element, except zero.

\begin{lem}\label{dcl}Let $\Gamma$ be a finitely generated algebraic closure in $R$.\begin{enumerate}\item If $\Gamma$ is non-trivial, it has an idempotent element $e$.
\item If $e$ is the only idempotent element in $\Gamma$, then $e\Gamma$ is a field.\end{enumerate}\end{lem}

\begin{proof}$(1)$ Suppose that $\Gamma$ is non zero, and let $r$ be in $\Gamma\setminus\{0\}$. By Corollary~\ref{rdecom}, there is a natural number~$n$ such that $R$ equals $r^nR\oplus Ann(R^n)$. In fact, one can show that $r^nR$ equals $r^{2n}R$ so the ring $R$ equals $r^{3n}R\oplus Ann(R^n)$. So, there is some $a$ in $Ann(r^n)$ and some $b$ such that $r^n=r^{3n}b+a$ holds. As the sum is direct, note that $a$ and $r^{3n}b$ must be in $\Gamma$. This yields $r^{2n}=r^{4n}b$. As $r$ cannot be nilpotent, $r^{2n}b$ is a non zero, and idempotent. As $r^{3n}b\in \Gamma$, it easily follows that $r^{2n}b\in\Gamma$. $(2)$ If $e$ is the only nilpotent element in $\Gamma$, one must have $r^{2n}b=e$, so $r$ is invertible in $eR$. As the inverse of $r$ is unique, it must be algebraic over $r$ and $e$. So $e\Gamma$ is a field.\end{proof}



\begin{prop}\label{wsar}In $R$, any finitely generated algebraic closure $\Gamma$ is isomorphic to a finite Cartesian product of fields.\end{prop}

\begin{proof}Suppose that $\Gamma$ is non zero. Then $\Gamma$ has a idempotent $e_1$ by Lemma~\ref{dcl}(1). We may choose it such that the additive group $e_1R\cap\Gamma$ is minimal among the groups $\{\gamma R\cap\Gamma:\gamma^2=\gamma\text{ and }\gamma\in\Gamma\setminus\{0\}\}$ by Theorem~\ref{dccm}. If $e_1$ has an orthogonal idempotent in $\Gamma$, we choose one, say $e_2$, such that $e_2R\cap\Gamma$ is minimal. If $e_1,e_2$ have a common orthogonal idempotent in $\Gamma$, we pick $e_3$ among the ones such that $e_3R\cap\Gamma$ is minimal. We claim that this process must stop. Otherwise, there would be an infinite chain $e_1,e_2,\dots$ of pairwise orthogonal idempotents. The chain $(Ann_\Gamma(e_1+\cdots+e_n))_{n\geq 1}$ would be strictly decreasing, a contradiction with the weakly small chain condition~\ref{dccm}. So let $e_1,\dots,e_n$ be such a family of pairwise orthogonal idempotents, of maximal size~$n$. Corollary~\ref{rdecom} yields  $$A=(e_1+e_2+\cdots+e_n)A\oplus Ann(e_1+\cdots+ e_n)$$ As the $e_i$ are pairwise orthogonal, one has
$$A=e_1A\oplus e_2A\oplus\cdots \oplus e_nA\oplus Ann(e_1+\cdots+ e_n)$$
Note that one has $(e_iR)\cap \Gamma=e_i\Gamma$ for every $i$. It follows that
$$\Gamma=e_1\Gamma\oplus e_2\Gamma\oplus\cdots \oplus e_n\Gamma\oplus Ann_\Gamma(e_1+\cdots +e_n)$$ By maximality of~$n$, the ring $Ann_\Gamma(e_1+\cdots+e_n)$ contains no idempotent. It must be zero by Lemma~\ref{dcl}(1). If for some $i$ the ring $e_i\Gamma$ should possess another idempotent $e_ia\neq e_i$, as $e_ia-e_i$ and $e_ia$ are pairwise orthogonal idempotents, one would have $e_ia\Gamma<e_i\Gamma$, a contradiction with the choice of $e_i$. So every $e_i\Gamma$ is a field by Lemma~\ref{dcl}(2).\end{proof}

\begin{rmq}Contrary to omega-stable rings, a small abelian ring with zero Jacobson radical need not necessary be isomorphic to a finite Cartesian product of fields. For instance, an infinite atomless boolean ring is $\aleph_0$-categorical \cite[Poizat, Théorème 6.21]{Poi1}. 
\end{rmq}




\subsection{Non abelian rings}Let $R$ be any weakly small ring.

Pb interessant a resoudre : si $R$ sans radical, sous quelle condition $\Gamma$ est-il sans radical?

\begin{lem}\label{lpa}Let $\Gamma$ be a finitely generated algebraic closure in~$R$.\begin{enumerate}
\item For every $r$ in $\Gamma$, there is some $a$ in $\Gamma$ and $n$ in $\mathbf N$ such that $r^{n}=r^{n}ar^{n}$.
\item If $\Gamma$ is not nil, it has at least one idempotent.\end{enumerate}\end{lem}

\begin{proof}$(1)$ As in the proof of Lemma~\ref{dcl}, for any $a$ in $\Gamma$, one can find some $b$ in $\Gamma$ such that $a^{n}=a^{2n}b$. By Corollary~\ref{rdecom} and symmetry, $\Gamma$ equals $\Gamma a^n\oplus Ann^\Gamma(a^n)$ so there are some $d$ in $\Gamma$ and $f$ in $Ann^\Gamma(a^n)$ with $b=da^n+f$. So $a^{2n}=a^{2n}da^{2n}$. $(2)$ If $\Gamma$ is not nil, one may choose $a$ non nilpotent, hence $a^{2n}d$ is non zero hence an idempotent element.\end{proof}

\begin{prop}\label{dern}For any finitely generated algebraic closure $\Gamma$ in $R$, the ring $\Gamma/J(\Gamma)$ is right artinian.\end{prop}

\begin{proof}Suppose that $\Gamma$ is non nil. Then it has an idempotent $e_1$ by Lemma~\ref{lpa}(2). We choose it so that the group $e_1\Gamma$ is minimal. If there is some idempotent $e_2$ with $e_1e_2=0$ and $e_2e_1\in J(\Gamma)$, we choose it such that $e_2\Gamma$ is minimal. Should this process not stop, there would be an infinite chain $e_1,e_2,\dots$ of such idempotent making the sequence $(Ann^\Gamma(e_1+\cdots+e_n))_{n\geq 1}$ strictly decreasing, a contradiction. So let be a maximal chain $e_1,\dots,e_n$ of idempotents in $\Gamma$ with $e_ie_j=0$ and $e_je_i\in J(\Gamma)$, and $e_i\Gamma$ minimal for every $i<j$. By Corollary~\ref{rdecom}, there is a natural number $m$ with $$\Gamma=((e_1+\cdots +e_n)^m)\Gamma\oplus Ann_\Gamma((e_1+\cdots +e_n)^m)$$
If $Ann_\Gamma(e_1+\cdots +e_n)$ has an idempotent $e$, then $ee_i=0$ for all $i$. Moreover, one can easily verify that $e_1e\Gamma$ is a nil right $\Gamma$-ideal. If follows from Fact~\ref{Jacob} that $e_1e\Gamma$ is included in $J(\Gamma)$. This contradicts the maximality of $n$, hence $Ann_\Gamma(e_1+\cdots+e_n)$ is a nil ideal by Lemma~\ref{lpa}(2). Note that $(e_1+\cdots +e_n)^m$ and $e_1+\cdots +e_n$ are equal modulo $J(\Gamma)$. This yields
$$\Gamma/J(\Gamma)=(\Gamma/J(\Gamma))e_1\oplus \cdots \oplus (\Gamma/J(\Gamma))e_n$$ 
Assume first that there be some natural number $i$, and some idempotent $e_ia$ in $e_i\Gamma$ such that $e_iae_i\neq e_i$. Then $e_i-e_iae_i$ and $e_iae_i$ are two orthogonal idempotents in $e_i\Gamma$, so $e_iae_i\Gamma<e_i\Gamma$, a contradiction with the choice of $e_i$. So for every $i$, every idempotent $e_ia$ in $e_i\Gamma$ verifies $e_iae_i=e_i$. It follows from Lemma~\ref{lpa}(1), that every element in $e_i\Gamma$ is either nil or $e_i$-invertible. Any right ideal of $e_i\Gamma$ must either equal $e_i\Gamma$, or be a nil ideal. Hence $\Gamma/J(\Gamma)$ is right Artinian.\end{proof}

\begin{cor}Let $\Gamma$ be a finitely generated algebraic closure in $R$. The ring $\Gamma/J(\Gamma)$ is isomorphic to a finite Cartesian product matrix rings over fields.\end{cor}

\begin{proof}By Proposition~\ref{dern}, $\Gamma/J(\Gamma)$ is right Artinian. Symmetrically, it must be left Artinian. The conclusion follows from Wedderburn-Artin's Theorem~\ref{WA}.\end{proof}


\end{document}